\documentclass[12pt]{article}

\usepackage[centertags]{amsmath}

\usepackage{amssymb}
\usepackage{amsthm}
\usepackage[misc]{ifsym}
\usepackage{mathrsfs}
\usepackage[colorlinks,linkcolor=blue]{hyperref}
\usepackage{authblk}
\usepackage{comment}

\theoremstyle{plain}
\newtheorem{thm}{Theorem}[section]
\newtheorem{cor}[thm]{Corollary}
\newtheorem{lem}[thm]{Lemma}
\newtheorem{prop}[thm]{Proposition}

\theoremstyle{definition}

\newtheorem{exam}[thm]{Example}
\newtheorem{rem}[thm]{Remark}

\newtheorem{conj}[thm]{Conjecture}

\theoremstyle{remark}
\numberwithin{equation}{section}

\newcommand{\overbar}[1]{\mkern 1.5mu\overline{\mkern-1.5mu#1\mkern-1.5mu}\mkern 1.5mu}

\begin{document}

\title{Nordhaus-Gaddum inequality for the spectral radius of a graph of order $n$}

\author[1]{Yen-Jen Cheng \textsuperscript{\Letter}}
\affil[1]{Department of Applied Mathematics, National Pingtung University, 4-18 Minsheng Road,
Pingtung, Taiwan}

\author[2]{Chih-wen Weng}
\affil[2]{Department of Applied Mathematics, National Yang Ming Chiao Tung University, 1001 University Road, Hsinchu, Taiwan}

\maketitle

\begin{abstract}
We determine the extremal graph $G$ of order $n$ that maximizes the sum of the spectral radii of $G$ and its complement. This resolves a conjecture posed by Stevanovi\'{c} in 2007.
\end{abstract}
\bigskip

\bigskip


\section{Introduction}

All the graphs in this paper are undirected and simple (i.e., without loops or multiple edges).
Let $G$ be a graph with vertex set $V(G)$ and edge set $E(G)$.   Unless otherwise specified, we assume $V=[n]:=\{1,2,\ldots,n\}$, where $n$ is the {\it order} of $G$. The {\it complement} of $G$, denoted by $\overbar{G}$, has vertex set  $V\left(\overbar{G}\right)=V(G)$ and $E\left(\overbar{G}\right)=\{ij~|~i, j\in V(G),~i\not=j, {\rm and}~ij\not\in E(G)\}.$ The {\it adjacency matrix} $A(G)=(a_{ij})$ of $G$ is the $n\times n$ matrix defined by $a_{ij}=1$ if $ij\in E(G)$ and $a_{ij}=0$ otherwise. The {\it spectral radius} of a nonnegative square matrix $M$, denoted by $\rho(M)$, is the maximum magnitude of the eigenvalues of $M$. The {\it spectral radius} $\rho(G)$ of a graph $G$ is defined as $\rho(A(G))$. The {\it join} of two vertex-disjoint graphs $G$ and $H$, denoted by $G\vee H$, is the graph obtained from $G$ and $H$ by adding all possible edges between $V(G)$ and $V(H)$. The join graph $K_q\vee N_{n-q}$, where $K_q$ is the complete graph of order $q$ and $N_{n-q}$ is the null graph of order $n-q$ is called a {\it complete split graph}. Its complement is the {\it disjoint union}  $N_q+ K_{n-q}$. The following conjecture was proposed by Stevanovi\'{c} in 2007.

\begin{conj}[\cite{abchrss:08,s:07}]\label{conj}
The maximum value of $\rho(G)+\rho\left(\overbar{G}\right)$ among graphs $G$ of order $n$ is attained by the complete split graph $K_{\lfloor\frac{n}{3}\rfloor}\vee N_{\lceil\frac{2n}{3}\rceil}$ and its complement. If $n\equiv 2 \ (\text{mod }  3)$, the maximum is also attained by $K_{\lceil\frac{n}{3}\rceil}\vee N_{\lfloor\frac{2n}{3}\rfloor}$ and its complement.
\end{conj}

Conjecture \ref{conj} is a Nordhaus-Gaddum type problem, which considers the maximum and minimum of the values $u(G)+u(\overbar{G})$ among all graphs of order $n$ for a given graph invariant $u$. In 1956, Nordhaus and Gaddum \cite{ng:56} proved the following inequalities for the chromatic number $\chi(G)$:
$$2\sqrt{n}\leq \chi(G)+\chi\left(\overbar{G}\right)\leq n+1,$$
$${n}\leq \chi(G)\cdot\chi\left(\overbar{G}\right)\leq \frac{(n+1)^2}{4}.$$
These results inspired a significant amount of subsequent research. See \cite{ah:13} for a comprehensive survey.

In 2007, Nikiforov \cite{n:07} proved that
$$\rho(G)+\rho\left(\overbar{G}\right)< \sqrt{2}(n-1)$$
and conjectured that
$$\rho(G)+\rho\left(\overbar{G}\right)\leq \frac{4}{3}n+O(1).$$
Csiv\'{a}ri \cite{c:09} improved Nikiforov's bound to $\frac{1+\sqrt{3}}{2}n-1$. Later, Terpai \cite{t:11} proved Nikiforov's conjecture using analytic method. Very recently, Liu \cite{l:24} proved Conjecture \ref{conj} for $n$ large enough using analytic method. In this paper, we prove Conjecture~\ref{conj} for every $n$ using only linear algebraic techniques.

\section{Preliminaries}
Conjecture~\ref{conj} is trivially true for $n\leq 2$. Hence we assume $n\geq 3$ throughout this paper. Let $\mathscr{S}(n)$ denote the set of $n\times n$ $(0,1)$-matrices with zero diagonal entries, and  $\mathscr{S}^*(n)$ be the subset of $\mathscr{S}(n)$ consisting of matrices $A=(a_{ij})$ satisfying $a_{12}=a_{21}=1$, $a_{n-1,n}=a_{n,n-1}=0$ and
$$\text{if $a_{ij}=1$, then $a_{hk}=1$ for $h\in [i]$ and $k\in [j]$ with  $k\ne h$.}$$
Let $O_n$, $J_n$ and $I_n$ denote the zero matrix, the all ones matrix and the identity matrix respectively of order $n$. Note that if one of $a_{12}=0$, $a_{21}=0$, $a_{n-1,n}=1$, $a_{n,n-1}=1$ holds, then
$A\ne \mathscr{S}^*(n)$ and $\rho(A)+\rho(J_n-I_n-A)=n-1$.
Let $\mathscr{S}_s^*(n)$ be the subset of  $\mathscr{S}^*(n)$ consisting of symmetric matrices.
Note that the adjacency matrices $A(K_n)$ and $A(\overbar{K_n})$ are not in  $\mathscr{S}_s^*(n)$. In this case $\rho(K_n)+\rho(\overbar{K_n})=n-1$ is less than the maximum value.
In this article, when we mention the {\it shape} of $A\in \mathscr{S}^*(n)$, we mean the staircase shape formed by the $1$'s in $A$.

The following is a necessary condition for a graph $G$ of order $n$ to achieve the maximum value of $\rho(G)+\rho\left(\overbar{G}\right)$ due to Csikv\'{a}ri \cite{c:09}.

\begin{lem}[\cite{c:09}]\label{lem2.1}
If $G$ is an extremal graph that attains the maximum value of $\rho(G)+\rho\left(\overbar{G}\right)$ among graphs of order $n$, then there exists an ordering of the vertex set such that $A(G)\in \mathscr{S}_s^*(n)$.
\end{lem}

\begin{rem}
While Csikv\'{a}ri \cite{c:09} established that the maximum value in Lemma~\ref{lem2.1} can be achieved by a graph $G$ with a vertex ordering satisfying $A(G)\in \mathscr{S}_s^*(n),$
 his proof implicitly demonstrates that this condition is essential. In other words, the maximum is attained only by such graphs.
\end{rem}

By Lemma~\ref{lem2.1}, to prove Conjecture~\ref{conj}, it suffices to consider graphs $G$ such that  $A(G)\in \mathscr{S}_s^*(n)$. Consequently, either $G$ or its complement graph $\overbar{G}$ is connected. Furthermore, if $G$ is connected  then $\overbar{G}$ is the disjoint union of a connected graph and some isolated vertices.

Although $A(G)$ is symmetric, we will sometimes need to consider nonsymmetric matrices.
Let $A$ be an $n\times n$ matrix, and let $r_i=r_i(A)$ denote the $i$-th row-sum of $A$. The vector $(r_1,r_2,\ldots,r_n)$ is called the {\it row-sum vector} of $A$. The following theorem will be a key tool in our proof of Conjecture \ref{conj}.

\begin{thm}[\cite{dz:13,lw:13}]\label{thm2.3}
If $A$ is a $(0,1)$-matrix with row-sum vector $(r_1,r_2,\ldots,r_n)$, where $r_1\geq r_2\geq\cdots\geq r_n$, then
for $1\leq \ell \leq n$,
\begin{equation}\label{eq2.1}\rho(A)\leq \phi_\ell(A):=\frac{1}{2}\left(
r_\ell-1+\sqrt{(r_\ell + 1)^2 + 4\sum_{i=1}^{\ell-1}(r_i-r_\ell)}\right).
\end{equation}
Moreover, if $A$ is symmetric and irreducible, then $\rho(A)=\phi_\ell(A)$
if and only if $r_1=r_n$ or there exists $t$ such that $2\leq t\leq \ell$, $r_1 = r_{t-1} = n-1$ and
$r_t = r_n$.
\end{thm}

\begin{rem}
Liu and Weng proved Theorem \ref{thm2.3} for symmetric $(0,1)$-matrices in \cite{lw:13}, while Duan and Zhou generalized it to nonsymmetric nonnegative matrices in \cite{dz:13}.
\end{rem}

Let $M=(m_{ab})$ be an $\ell\times \ell$ matrix with row-sum vector $(r_1,r_2, \ldots,r_\ell)$. Then  $M$ is {\it rooted} if there exists a constant $d$ such that for $a\in [\ell]$ and $b\in [\ell-1]$, we have $d\delta_{ab}+m_{ab}\geq m_{\ell b}\geq 0$ and $r_a+d\geq r_\ell+d\geq 0$, where $\delta_{ab}=1$ if $a=b$ and $\delta_{ab}=0$ otherwise. Let $\rho_r(M)$ denote the largest real eigenvalue of $M$.
In a particular case of the proof of Conjecture \ref{conj}, we will use the following generalization of Theorem~\ref{thm2.3} due to Cheng and Weng \cite{cw:23}.

\begin{thm}[{\cite[Theorem 5.1]{cw:23}}]\label{thm2.5}
Let $M=(m_{ab})$ be an $\ell\times \ell$ rooted matrix.  If $C=(c_{ij})$ is an $n\times n$ nonnegative matrix and there exists a partition $\Pi=\{\pi_1, \pi_2, \ldots, \pi_\ell\}$ of $[n]$ such that
$$\max_{i\in \pi_a} \sum_{j\in \pi_b} c_{ij}\leq m_{ab} \quad \hbox{and} \quad \max_{i\in \pi_a} \sum_{j=1}^n c_{ij}\leq \sum_{c=1}^\ell m_{ac}$$
for $a\in [\ell]$ and $b\in [\ell-1]$, then $\rho(C)\leq \rho_r(M).$
\end{thm}

We will use the concept of an equitable partition and its associated quotient matrix to simplify the computation of $\rho_r(M)$ (see, e.g., \cite[Section 2.3]{Brou}). Let $\Pi=\{\pi_1,\ldots,\pi_k\}$ a partition of $[\ell]$ into $k$ blocks. For $s, t\in [k]$, let
$M_{st}$ denote the submatrix of $M$ induced on $\pi_s\times \pi_t$.
If the row sums of each submatrix $M_{st}$  are constant, say  $m'_{st}$, then $\Pi$ is called the {\it equitable partition} of $M$, and the $k\times k$ matrix $\Pi(M)=(m'_{st})_{s, t\in [k]}$ is called the {\it quotient matrix} of $M$ with respect to equitable partition $\Pi$. The following two lemmas are well-known results.

\begin{lem}[\cite{cw:23}]\label{lem2.6}
If $M$ is a rooted $\ell\times \ell$ matrix and its transpose $M^T$ has an equitable partition
$\Pi=\{\pi_1, \pi_2, \ldots,\pi_k\}$ of $[\ell]$  such that $\pi_k=\{\ell\}$, then $\rho_r(M)=\rho_r(\Pi(M^T))$.
\end{lem}

\begin{lem}[\cite{g:93}]\label{lem2.7}
If $M$ is a nonnegative $\ell\times \ell$ matrix that has an equitable partition $\Pi=\{\pi_1,\pi_2,\ldots,\pi_k\}$ of $[\ell]$, then $\rho(M)=\rho(\Pi(M))$.
\end{lem}

The following lemma is useful for estimating upper bounds of $\rho(G)+\rho\left(\overbar{G}\right)$ later.
\begin{lem}\label{lem2.8}
If $0\leq E \leq F$ and $E + F$ is a constant $C$, then $\sqrt{E}+\sqrt{F}$ is increasing as a function of $E$, and the maximum occurs when $E = C/2$ (i.e., $E = F$).
\end{lem}

\begin{proof}
Note that $\left(\sqrt{E}+\sqrt{F}\right)^2=\left(\sqrt{E}+\sqrt{C-E}\right)^2=C+2\sqrt{E(C-E)}$. Then $E(C-E)$ is increasing on the interval $[0, C/2]$ and decreasing on the interval $[C/2, C]$ and the result follows.
\end{proof}

Let $A=(a_{ij})$ be an $m\times n$ matrix and $B=(b_{ij})$ be a $p\times q$ matrix. The {\it tensor product} of $A$ and $B$, denoted by $A\otimes B$, is the $mp\times nq$ block matrix
$$\begin{pmatrix}
a_{11}B&a_{12}B&\cdots&a_{1n}B\\
a_{21}B&a_{22}B&\cdots&a_{2n}B\\
\vdots&\vdots&\cdots&\vdots\\
a_{m1}B&a_{m2}B&\cdots&a_{mn}B
\end{pmatrix}.$$
For column vectors  $u$ and $v$, we define $u\otimes v$ by regarding $u$ and $v$ as matrices with  one column. The following is a fundamental result in matrix theory. One may follow the proof of \cite[Theorem 4.2.12]{Horn}) to prove the result.

\begin{lem}[{\cite[Theorem 4.2.12]{Horn}}]\label{lem2.9}
If $A$ and $B$ are $n\times n$ and $m\times m$ matrices with eigenvalues $\lambda_1, \lambda_2, \ldots, \lambda_n$ and $\mu_1, \mu_2, \ldots, \mu_m$, respectively, then the eigenvalues of $A\otimes I_m+I_n\otimes B$ are $\lambda_i+\mu_j$ for $1\leq i,j\leq n$.
\end{lem}

\begin{proof}  Note that $A\otimes I_m+I_n\otimes B$ is an $nm\times nm$ matrix.
Let $u_i$ be a generalized $\lambda_i$-eigenvector of $A$, and $v_j$ be a generalized $\mu_j$-eigenvector of $B$
such that  $\{u_i\mid i\in [n]\}$ and $\{v_j\mid j\in [m]\}$ are bases in $\mathbb{R}^n$ and $\mathbb{R}^m$, respectively. Then there exists $\ell$ and $k$ suck that $(A-\lambda_iI_n)^\ell u_i=0$ and $(B-\mu_jI_m)^kv_j=0$. The lemma follows since 
$$
\begin{array}{ll}
&(A\otimes I_m+I_n\otimes B-(\lambda_i+\mu_j)I_n\otimes I_m)^{\ell+k}(u_i\otimes v_j)\\
=&((A-\lambda_iI_n)\otimes I_m+I_n\otimes (B-\mu_j I_m))^{\ell+k}(u_i\otimes v_j)\\
=&\displaystyle\sum_{h=0}^{\ell+k}{\ell+k\choose h}(A-\lambda_iI_n)^hu_i\otimes(B-\mu_j I_m)^{\ell+k-h}v_j\\
=&0
\end{array}$$ 
and  $\{u_i\otimes v_j\mid i\in [n],~j\in [m]\}$ is a basis  of $\mathbb{R}^{nm}$.
\end{proof}

\section{The upper bound $\phi(A)$ of $\rho(A)$}\label{sec3}

For $A\in\mathscr{S}^*(n)$, there are $n$ upper bounds $\phi_\ell(A)$  of $\rho(A)$ with $\ell\in [n]$ in Theorem~\ref{thm2.3}.  We will choose a particular nice one in this section. Indeed, we will choose
$\phi_\ell(A)$ with $\ell=c+1$, where $$c=c(A):=\max\left\{i\in [n]~:~ r_1+r_2+\cdots+r_i>  i(i-1)\right\}$$
and $(r_1\ldots,r_n)$ is the row-sum vector of $A$.
The above $c$ is well-defined since $r_1> 0=1\cdot(1-1)$ and $r_1+r_2+\cdots+r_n< n(n-1).$
In particular $c\in [n-1]$.
Let $v=v(A):=r_{c+1}$ and  $s=s(A):=\sum_{i=1}^{c}r_i-c(c-1)$.
For example, the matrix $A(K_t+ N_{n-t})$ has $c=v=s=t-1$, where $1\leq t\leq n-1$.
For another example, the matrix $A$ in the left of Figure~\ref{fig1} has  $c=4$, $v=2$ and  $s=1$.

\begin{figure}[h]
$$
A=\begin{array}{rl}
& \left.\begin{array}{cccccc} ~~&~~ v&~&~~~&~&~\end{array}\right.  \\
\begin{array}{r}
1\\
\downarrow\\
\\
c\\
\\
\\
\end{array}
&
\left(
\begin{array}{cccccc}
0&1&1&\multicolumn{1}{c|}{1} &1&1\\
1&0&1&\multicolumn{1}{c|}{1}&0&0\\ \cline{3-6}
1&1&\multicolumn{1}{|c}{0}&\multicolumn{1}{c|}{0}&0&0\\
1&1&\multicolumn{1}{|c}{1}&\multicolumn{1}{c|}{0}&0&0 \\ \cline{1-4}
1&1&\multicolumn{1}{|c}{0}&0&0&0\\
1&0&\multicolumn{1}{|c}{0}&0&0&0
\end{array}\right)\\
&\begin{array}{cccc} ~~&~~&~~\overbar{c}&~~\leftarrow~1\end{array}
\end{array}, ~
A_1=\begin{array}{rl}
& \left.\begin{array}{cccccc} ~~&~~&~v_1&~~~&~&~\end{array}\right.  \\
\begin{array}{r}
\\
\\
\\
c_1\\
\\
\\
\end{array}
&
\left(
\begin{array}{cccccc}
0&1&1&\multicolumn{1}{c|}{1} &1&1\\
1&0&1&\multicolumn{1}{c|}{1}&0&0\\
1&1&0&\multicolumn{1}{c|}{0}&0&0\\
1&1&1&\multicolumn{1}{c|}{0}&0&0 \\ \cline{1-4}
1&1&1&0&0&0\\
0&0&0&0&0&0
\end{array}\right)
\end{array}
$$

\caption{The matrices $A, A_1\in \mathscr{S}^*(n)$ with $\rho(A)\leq \phi(A)<  \phi(A_1)$.}
\label{fig1}
\end{figure}

Since $c, v, s$ are uniquely determined from $A$, we denote the upper bound $\phi_{c+1}(A)$ of $\rho(A)$ in Theorem~\ref{thm2.3} by $\phi(A)$, where
\begin{align}\label{eq3.1}
\phi(A):=&\frac{1}{2}\left(v-1+\sqrt{(v + 1)^2 + 4(s+c(c-1)-cv)}\right) \nonumber\\
           =&\frac{1}{2}\left(v-1+\sqrt{(2c-v-1)^2+4s}\right).
\end{align}

\begin{lem}\label{lem3.1}
If $A\in\mathscr{S}^*(n)$ has parameters $c=c(A)$, $v=v(A)$, $s=s(A)$ and the function $\phi(A)$ of $A$ is as in (\ref{eq3.1}), then the following (i)-(iii) hold.
\begin{enumerate}
\item[(i)] $\rho(A)\leq \phi(A)$. Moreover, if $A\in\mathscr{S}^*_s(n)$,  then $\rho(A)= \phi(A)$ if and only if  $A=A(K_{r_1+1}+N_{n-r_1-1})$ or there exists $t$ such that $2\leq t\leq c+1$ and $A=A((K_{t-1}\vee N_{r_1+2-t})+N_{n-r_1-1}).$
\item[(ii)] $v\leq c$ and $0<s\leq 2c-v$.
\item[(iii)]$c-1<\phi(A)\leq c$, and $\phi(A)=c$ if and only if $s+v=2c$.
\end{enumerate}
\end{lem}

\begin{proof}
(i) The first statement is a rephrasing of (\ref{eq2.1}) by choosing $\ell=c+1$.
To prove the second statement, assume $A\in\mathscr{S}^*_s(n)$. Let $A'\in \mathscr{S}^*_s(r_1+1)$ be the symmetric and irreducible matrix obtained from $A$ by deleting the zero rows and columns. Since the row-sum vector of $A'$ is the same as the nonzero part of the row-sum vector of $A$,
$\rho(A)=\phi(A)$ is equivalent to $r_1=r_{r_1+1}$ or there exists $t$ such that $2\leq t\leq c+1$, $r_1=r_{t-1}$ and $r_t=r_{r_1+1}$ by the second statement of Theorem~\ref{thm2.3}. This is also equivalent to $A=A(K_{r_1+1}+N_{n-r_1-1})$ or there exists $t$ such that $2\leq t\leq c+1$ and $A=A((K_{t-1}\vee N_{r_1+2-t})+N_{n-r_1-1})$ since $A$ is symmetric.
\medskip

\noindent (ii) If $v\geq c+1$ then from the shape of $A\in \mathscr{S}^*(n)$, we have $\sum_{j=1}^{c+1} r_j \geq (c+1)^2>c(c+1)$, a contradiction to the definition of $c$. Since
$\sum_{j=1}^c r_j> c(c-1)$ and $\sum_{j=1}^{c+1} r_j\leq c(c+1),$ we have $s=\sum_{j=1}^c r_j-c(c-1)>0$ and $2c-v=c(c+1)-c(c-1)-v\geq \sum_{j=1}^{c+1} r_j-c(c-1)-r_{c+1}=s$.
\medskip

\noindent (iii) By (\ref{eq3.1}) and (ii) above, we have
\begin{align*}
\phi(A)&=\frac{1}{2}\left(v-1+\sqrt{(2c-v-1)^2+4s}\right) \\
&\leq\frac{1}{2}\left(v-1+\sqrt{(2c-v-1)^2+4(2c-v)}\right)=c.
\end{align*}
Moreover, $\phi(A)=c$  if and only if $s=2c-v$. Since $s>0$ by (ii), we have  $\phi(A)> \frac{1}{2}(v-1+2c-v-1)=c-1$ by (\ref{eq3.1}).
\end{proof}

The following is a realization of the value $\phi(A)$ by the spectral radius of a matrix
\begin{equation}\label{eq3.2}\phi(A)=\rho\left(\begin{bmatrix}c-1& s \\ 1& v-c\end{bmatrix}\right)=\rho\left(\begin{bmatrix}2c-1& s \\ 1& v\end{bmatrix}\right)-c.\end{equation}

\begin{prop}\label{prop3.2}
If $A\in\mathscr{S}_s^*(n)$ has parameters $c=c(A)$, $v=v(A)$, $s=s(A)$
and there exists $A_1\in \mathscr{S}^*(n)$ whose parameters $c_1=c(A_1)$, $v_1=v(A_1)$, $s_1=s(A_1)$ satisfy $c_1=c$, $v_1>v$ and $s_1\geq s$, then $\rho(A)<\rho(A_1)$.
\end{prop}
\begin{proof}
Since the second $2\times 2$ matrix in (\ref{eq3.2}) is irreducible, increasing
$v$ or $s$  while keeping the remaining three values unchanged will increase its spectral radius. So the lemma follows.
\end{proof}

For the matrix $A$ in the left of Figure~\ref{fig1}, we have $\rho(A)\leq \phi(A)$ by Lemma~\ref{lem3.1}(i). If we interchange values $a_{6,1}=1$ and $a_{5,3}=0$ of $A$ then we obtain the matrix $A_1\in \mathscr{S}^*(n)$ with $c_1=c(A_1)=4=c(A)$, $v_1=v(A_1)=3=v(A)+1$ and $s_1=s(A_1)=1=s(A)$, and
$\phi(A_1)>\phi_(A)$  by Proposition~\ref{prop3.2}.

For $A\in\mathscr{S}^*(n)$, let $\overbar{A}=(\overbar{a}_{ij})$ be the reflection of the matrix $J_n-I_n-A$ about the anti-diagonal line, i.e., $\overbar{A}$ is an $n\times n$ matrix with zero diagonal entries, and for distinct $i, j\in [n]$,
\begin{equation}\label{eq3.3}\overbar{a}_{ij}=1-a_{n-j+1, n-i+1}.\end{equation} This definition is slightly different to one might expect.
One reason is to let $\overbar{A}\in \mathscr{S}^*(n).$
The integral values $c=c(A), v=v(A), s=s(A), \overbar{c}=c(\overbar{A}), \overbar{v}=v(\overbar{A}), \overbar{s}=s(\overbar{A})$ are referred to as  the {\it parameters} of $A$.
We have seen the three parameters $(c, v, s)=(4, 2, 1)$ of $A$ in Figure~\ref{fig1}.  We try to find the remaining three parameters $(\overline{c}, \overline{v}, \overline{s})$ of $A$  without explicitly writing down the complement $\overline{A}$ of $A$ by reading $A$ leftward from the $a_{6, 6}$ position as row increasing of $\overline{A}$, and switching the role of the off-diagonal $0$'s and $1$'s of $A$. Since the zero in $a_{2, 5}$ will replace the one in $a_{4,3}$ while one zero in $a_{2, 6}$  still remains in the upper right corner entries $(i, j)$ with $1\leq i\leq 2$ and $3\leq j\leq 6$, we find $\overbar{c}=4$ and $\overline{s}=1$.  Since $a_{6, 2}=0$ and $a_{5, 2}=1$, we find $\overline{v}=1$.
Note that if we change the values in the submatrix of $A$ restricted to $\{5, 6\}\times \{1, 2\}$,
then the values $v$ and $\overbar{v}$ might change, but the values $c$, $\overbar{c}$, $s$ and $\overbar{s}$ remain the same. This is another reason for the above definition of $\overbar{A}$.

Note that if $a_{c+1,n-\overbar{c}}=1$, then $v\geq n-\overbar{c}$ and $\overbar{v}\leq n-c-1$; if $a_{c+1,n-\overbar{c}}=0$, then $\overbar{v}\geq n-c$ and $v\leq n-\overbar{c}-1$.

\begin{lem}\label{lem3.3}
If $A\in\mathscr{S}_s^*(n)$ has parameters $c=c(A)$, $v=v(A)$, $s=s(A)$,
$\overbar{c}=c(\overbar{A})$, $\overbar{v}=v(\overbar{A})$, $\overbar{s}=s(\overbar{A})$
with $c+\overbar{c}\geq n$ and $\overbar{v}<\min\{2\overbar{c}-\overbar{s},n-c-1\}$,  then
there exists $A_1\in \mathscr{S}^*(n)$ whose parameters $c_1=c(A_1)$, $v_1=v(A_1)$, $s_1=s(A_1)$, $\overbar{c}_1=c(\overbar{A}_1)$, $\overbar{v}_1=v(\overbar{A}_1)$, $\overbar{s}_1=s(\overbar{A}_1)$
satisfy $c_1=c$, $v_1=v$, $s_1=s$, $\overbar{c}_1=\overbar{c}$, $\overbar{v}_1=\min\{2\overbar{c}_1-\overbar{s}_1,n-c_1-1\}$ and $\overbar{s}_1\geq \overbar{s}$. Moreover, $\phi(A)=\phi(A_1)$ and $\phi(\overbar{A})<\phi(\overbar{A}_1)$.
\end{lem}
\begin{proof}
The matrix $A_1$ is obtained from $A$ by adding $\overbar{v}_1-\overbar{v}$ more $0$'s into the positions $(i, n-\overbar{c})$ for $n-\overbar{v}\geq i\geq c+2$, and if necessary, adding $0$'s into the positions $(i,j)$ for $i\geq c+2,j\geq n-\overbar{c}+1$ to keep $A_1\in \mathscr{S}^*(n)$.
This process will increase $\overbar{v}_1,\overbar{s}_1$, keep $c_1=c$, $v_1=v$ and $s_1=s$, and keep $\overbar{c}_1=c$ until $\overbar{v}_1=2\overbar{c}_1-\overbar{s}_1$ by the definition of $c,v,s,\overbar{c}$. Conduct the process until $\overbar{v}_1=\min\{2\overbar{c}_1-\overbar{s}_1,n-c_1-1\}$. Then $\phi(A)=\phi(A_1)$ and $\phi(\overbar{A})<\phi(\overbar{A}_1)$ by Proposition \ref{prop3.2}.
\end{proof}

From Lemma \ref{lem3.3}, sometimes we might assume $\overbar{v}=\min\{2\overbar{c}-\overbar{s},n-c-1\}$, and adjust $v$ to get a larger $\phi(A)+\phi(\overbar{A})$. Lemma \ref{lem3.4} and Lemma \ref{lem3.5} deal with the case $\overbar{v}=n-{c}-1$ and the case $\overbar{v}=2\overbar{c}-\overbar{s}$, respectively.

\begin{lem}\label{lem3.4}
If $A\in\mathscr{S}^*(n)$ has parameters $c=c(A)$, $v=v(A)$, $s=s(A)$,
$\overbar{c}=c(\overbar{A})$, $\overbar{v}=v(\overbar{A})$, $\overbar{s}=s(\overbar{A})$ satisfying $n-\overbar{c}\leq v'< v$ and $\overbar{v}=n-c-1$,  then there exists $A_2\in \mathscr{S}^*(n)$ whose parameters $c_2=c(A_2)$, $v_2=v(A_2)$, $s_2=s(A_2)$, $\overbar{c}_2=c(\overbar{A}_2)$, $\overbar{v}_2=v(\overbar{A}_2)$, $\overbar{s}_2=s(\overbar{A}_2)$
satisfy $c_2=c$, $v_2=v'$, $s_2+v_2=s+v$, $\overbar{c}_2=\overbar{c}$, $\overbar{v}_2=\overbar{v}$, $\overbar{s}_2=\overbar{s}$. Moreover, $\phi(A)\leq \phi(A_2)$,  $\phi(\overbar{A})= \phi(\overbar{A}_2)$, and 
 $\phi(A)=\phi(A_2)$  if and only if $v+s=2c$.
\end{lem}

\begin{proof}
To prove the lemma, it suffices to assume $v'=v-1$.
By Lemma \ref{lem3.1}(ii), $v\leq c$. From $n-\overbar{c}<v\leq c$, we have $c+\overbar{c}>n$. Then $\overbar{c}\geq 2$ since $c\leq n-1$, and $a_{c,n}=0$, otherwise $c+\overbar{c}\leq n$, a contradiction.
Let $A_2$ be a matrix obtained from $A$ by  moving the $1$ in the position $(c+1, v)$ to the upper right region while keeping $A_2\in \mathscr{S}^*(n).$ From the definition of $c$, we have $c_2=c$, $v_2=v'$ and $s_2+v_2=s+v$. The assumption $n-\overbar{c}\leq v'< v$ implies $\overbar{c}_2=\overbar{c}$, $\overbar{v}_2=\overbar{v}$ and $\overbar{s}_2=s$.

For the second part,
let $$g(x)=x^2-(v-1)x+(v-c)(c-1)-s,$$
$$g_2(x)=x^2-(v_2-1)x+(v_2-c)(c-1)-s_2.$$
Then $\phi(A)$ and $\phi(A_2)$ are the largest zero of $g(x)$ and $g_2(x)$, respectively. Note that
$$g_2(x)-g(x)=(v-v_2)x-(v-v_2)(c-1)+s-s_2$$
and $\phi(A)\leq c$ by Lemma \ref{lem3.1}(iii).
Then
$$g_2(\phi(A))\leq (v-v_2)\cdot c-(v-v_2)(c-1)+s-s_2=0$$
and $\phi(A)\leq \phi(A_2)$. If the equality holds, then $\phi(A)=c$ and $v+s=2c$ by Lemma \ref{lem3.1}(iii). If $v+s=2c$, then $\phi(A)=\phi(A_2)=c$ by Lemma \ref{lem3.1}(iii) again.
\end{proof}

In Figure~\ref{fig2} the matrix $A_2$ is obtained from $A$ by interchanging the values $a_{5,4}=1$
and $a_{1, 6}=0$. One can check from the definitions that  $A$ and $A_2$  have parameters $c=4=c_2$, $v=4=v_2+1$, $s=1=s_2-1$, $\overbar{c}=3=\overbar{c}_2$, $\overbar{v}=1=\overbar{v}_2$, $\overbar{s}=4=\overbar{s}_2$  and
$\phi(A)<\phi(A_2)$ by Lemma~\ref{lem3.4}.

\begin{figure}[h]
$$
A=\begin{array}{rl}
& \left.\begin{array}{cccccc} ~~&~~&~~&~v&~&~\end{array}\right.  \\
\begin{array}{r}
1\\
\downarrow\\
\\
c\\
\\
\\
\end{array}
&
\left(
\begin{array}{lccccc}
0&1&1&\multicolumn{1}{c|}{1} &1&0\\
1&0&1&\multicolumn{1}{c|}{1}&0&0\\
1&1&0&\multicolumn{1}{c|}{1}&0&0\\  \cline{4-6}
1&1&1&\multicolumn{1}{|c|}{0}&0&0 \\ \cline{1-4}
1&1&1&\multicolumn{1}{|c}{1}&0&0\\
0&0&0&\multicolumn{1}{|c}{0}&0&0
\end{array}\right) \\
&\begin{array}{cccc} ~~&~~&~~&~\overbar{c}~\leftarrow~1\end{array}
\end{array},~  A_2=\begin{array}{rl}
& \left.\begin{array}{cccccc} ~~&~~&~ v_2&~~~&~&~\end{array}\right.  \\
\begin{array}{r}
1\\
\downarrow\\
\\
c_2\\
\\
\\
\end{array}
&
\left(
\begin{array}{lccccc}
0&1&1&\multicolumn{1}{c|}{1} &1&1\\
1&0&1&\multicolumn{1}{c|}{1}&0&0\\
1&1&0&\multicolumn{1}{c|}{1}&0&0\\  \cline{4-6}
1&1&1&\multicolumn{1}{|c|}{0}&0&0 \\ \cline{1-4}
1&1&1&\multicolumn{1}{|c}{0}&0&0\\
0&0&0&\multicolumn{1}{|c}{0}&0&0
\end{array}\right) \\
&\begin{array}{cccc} ~~&~~&~~&~\overbar{c}_2~\leftarrow~1\end{array}
\end{array}
$$

\caption{The matrices $A, A_2$ with $c(A)=c(A_2)=4$ and $\phi(A)<\phi(A_2)$.}
\label{fig2}
\end{figure}

\begin{lem}\label{lem3.5}
If $A\in\mathscr{S}^*(n)$ has parameters $c=c(A)$, $v=v(A)$, $s=s(A)$,
$\overbar{c}=c(\overbar{A})$, $\overbar{v}=v(\overbar{A})$, $\overbar{s}=s(\overbar{A})$ satisfying $n-\overbar{c}< v$ and $\overbar{v}=2\overbar{c}-\overbar{s}$,  then there exists $A_2\in \mathscr{S}^*(n)$ whose parameters $c_2=c(A_2)$, $v_2=v(A_2)$, $s_2=s(A_2)$, $\overbar{c}_2=c(\overbar{A}_2)$, $\overbar{v}_2=v(\overbar{A}_2)$, $\overbar{s}_2=s(\overbar{A}_2)$
satisfy $c_2=c$, $v_2=\max\{2c_2-s_2,n-\overbar{c}\}$, $s_2+v_2\geq s+v$, $\overbar{c}_2=\overbar{c}$, $\overbar{v}_2=\overbar{v}$, $\overbar{s}_2=\overbar{s}$. Morover, $\phi(A)\leq \phi(A_2)$ and $\phi(\overbar{A})\leq \phi(\overbar{A}_2)$.
\end{lem}
\begin{proof}
Similar to the proof of Lemma \ref{lem3.4}, we have $c+\overbar{c}\geq n$, $\overbar{c}\geq 2$ and $a_{c,n}=0$.
Let $A_2$ be a matrix obtained from $A$ by moving the $1$'s in the positions $(i, v)$ for $i\geq c+1$ to the upper right region while keeping $A_2\in \mathscr{S}^*(n)$ until $v_2=2c_2-s_2$. From the definition of $c$, we have , $c_2=c$, $s_2\geq s$ and $s_2+v_2\geq s+v$. Additionally, $v_2=v-1$ if $v_2=2c_2-s_2$ does not hold. By Continuing this process, the result follows. The proof of the second part is similar to that of Lemma \ref{lem3.4}.
\end{proof}

\section{The complete split graphs $K_q\vee N_{n-q}$}

We study complete split graphs $K_q\vee N_{n-q}$ and their associated parameters in this section.

\begin{lem}\label{lem4.1}
Let $G$ be a graph whose adjacency matrix $A\in\mathscr{S}^*_s(n)$ has parameters $(c, v, s, \overbar{c}, \overbar{v}, \overbar{s})$.
If $\rho(A)+\rho(\overbar{A})= \phi(A)+\phi(\overbar{A})$,   then $G=K_q\vee N_{n-q}$ or $\overbar{G}=K_q\vee N_{n-q}$, and
\begin{equation}\label{eq4.1}\{ \overbar{c}, c\}=\left\{n-q-1,~\left\lceil\frac{q-1+\sqrt{(q-1)^2+4q(n-q)}}{2}\right\rceil\right\}.
\end{equation}
\end{lem}
\begin{proof}
The assumption $\rho(A)+\rho(\overbar{A})= \phi(A)+\phi(\overbar{A})$ implies  $\rho(A)= \phi(A)$ and $\rho(\overbar{A})=\phi(\overbar{A})$ by Lemma~\ref{lem3.1}(i).
By interchanging $A$ and $\overbar{A}$ if necessary,  we might assume the first row sum of $A$ is $r_1=n-1$. By Lemma~\ref{lem3.1}(i) and the fact $J_n-I_n\not\in\mathscr{S}^*_s(n)$, there exists $2\leq t\leq c+1$ such that $r_1=r_{t-1}$
and $r_t=r_n$. This is equivalent to $G=K_q\vee N_{n-q}$ with $q=t-1$, and $\overbar{G}=N_q+K_{n-q}$. The latter implies $\overbar{c}=n-q-1$, and the former implies that
 $c$ is the largest integer with
$$
0<\sum_{i=1}^c (r_i-c+1)=(n-c)q+(q-c+1)(c-q)=-c^2+(q+1)c+q(n-q-1),$$ which implies that $c$ is  the other term in the set in (\ref{eq4.1}).
\end{proof}

Following the above proof, all the parameters of $A(K_q\vee N_{n-q})$ are listed below.
\begin{exam}\label{exam4.2}
Let  $G=K_q\vee N_{n-q}$ be the complete split graph with $q\in [n-2]$ and $A=A(G)\in \mathscr{S}^*_s(n)$. Then the condition for the equality in Theorem \ref{thm2.3} holds with $\ell=q+1$.
Hence \begin{equation}\label{eq4.2}\rho(A)=\frac{q-1+\sqrt{(q-1)^2+4q(n-q)}}{2},\quad \rho(\overbar{A})=n-q-1,\end{equation}
and $A$ has parameters $c=\lceil \rho(A)\rceil$, $v=q$, $s=(n-c)q+(q-c+1)(c-q)$, and  $\overbar{c}=\overbar{v}=\overbar{s}=n-q-1$.
\end{exam}

Let $\rho_0:=\rho(K_{\lfloor\frac{n}{3}\rfloor}\vee N_{n-\lfloor\frac{n}{3}\rfloor})$. As a remark, Aouchiche et al. \cite{abchrss:08} proved that the largest value of $\rho(G)+\rho(\overbar{G})$, where $G=K_q\vee N_{n-q}$, is attained if and only if $q=\left\lfloor \frac{n}{3}\right\rfloor$, and also
$q=\lceil \frac{n}{3}\rceil$ if $n\equiv 2\pmod 3$,
i.e., $G$ is the extremal graph in Conjecture \ref{conj}.

\begin{prop}\label{prop4.3} Let $G=K_q\vee N_{n-q}$ then $\rho_0=\rho(G)+\rho(\overbar{G})$ if $q=\left\lfloor \frac{n}{3}\right\rfloor$, and also $q=\lceil \frac{n}{3}\rceil$ if $n\equiv 2\pmod 3$. Indeed,
\begin{align}
\rho_0=& \frac{1}{2} \left( 2n-3-\left\lfloor \frac{n}{3} \right\rfloor + \sqrt{\left(\frac{2n-1}{3} + \left\lfloor \frac{n}{3} \right\rfloor \right)^2+ \frac{8k_n}{9}}\right) \label{eq4.3}\\
\geq& \frac{4n-5}{3}, \label{eq4.4}
\end{align}
where  $k_n=0$ if $n\equiv 2\pmod 3$ and $k_n=1$ otherwise.
\end{prop}

\begin{proof} For $A=A(K_q\vee N_{n-q})$, computing  from (\ref{eq4.2}) according to each case $n=3k$, $3k+1$ or $3k+2$, we get the first part of
the result and \eqref{eq4.3}. The inequality in $(\ref{eq4.4})$ is obtained from (\ref{eq4.3}) by replacing the nonnegative integer $k_n$ by $0$.
\end{proof}

\section{The upper bound $c+\overbar{c}$}

Let $A=(a_{i,j})\in \mathscr{S}^*(n)$ with parameters $c=c(A)$, $v=v(A)$, $s=s(A)$, $\overbar{c}=c(\overbar{A}), \overbar{v}=v(\overbar{A}), \overbar{s}=s(\overbar{A})$. Combining Lemma~\ref{lem3.1} (i) and (iii), we have $\rho(A)+\rho(\overbar{A})\leq c+\overbar{c}$.
In this section, we study this upper bound $c+\overbar{c}$ of $\rho(A)+\rho(\overbar{A})$. It turns out that only two situations $c+\overbar{c}=\lfloor \frac{4n}{3} \rfloor$ and
$c+\overbar{c}=\lfloor \frac{4n}{3} \rfloor-1$ are of interest. In the end of this section, we prove Conjecture~\ref{conj} in the special case $(n, c+\overbar{c})=(3k+2, 4k+1).$
\medskip

Since $\overbar{a}_{\overbar{c}+1, n-c}=1-a_{c+1, n-\overbar{c}}$ by (\ref{eq3.3}),
by interchanging $A$ and $\overbar{A}$ if necessary, one might assume $a_{c+1, n-\overbar{c}}=1$ or equivalently $v\geq n-\overbar{c}$, or $\overbar{v}\leq n-c-1$, which will appear often thereafter.

\begin{lem}\label{lem5.1} With the notations above, let $c+\overbar{c}\geq n$. Then
\begin{equation}\label{eq5.1}
s+\overbar{s}\leq -\frac{3}{4}(c+\overbar{c})^2+(n+1)(c+\overbar{c})-\frac{1}{4}(c-\overbar{c})^2-v-\overbar{c},
\end{equation}
with the equality if $\overbar{v}=n-c-1$.
Moreover, we have
\begin{equation}\label{eq5.2}c+\overbar{c}<\frac{4n}{3}+\frac{1}{3}.\end{equation}
\end{lem}
\begin{proof}
From the construction, the number of $1$'s (resp. of off-diagonal $0$'s) in the first $c+1$ rows (resp. the last $\overbar{c}$ columns) of $A$ is $c(c-1)+s+v$ (resp.
$\overbar{c}(\overbar{c}-1)+\overbar{s}$).
There are $n(n-1)-(n-c-1)(n-\overbar{c})$ off-diagonal entries in the first $c+1$ rows and in the last $\overbar{c}$ columns of $A$.  Hence we have
\begin{equation}\label{eq5.4}
c(c-1)+s+v+ \overbar{c}(\overbar{c}-1)+\overbar{s}\leq  n(n-1)-(n-c-1)(n-\overbar{c}).
\end{equation}
Deleting common terms in both sides of (\ref{eq5.4}) and using
$((c+\overbar{c})^2-(c-\overbar{c})^2)/4$ to replace $c\overbar{c}$, we immediately have (\ref{eq5.1}).
Since $c+\overbar{c}\geq n$, only an off-diagonal $0$'s  in the positions $(i, j)$ with $i\leq c+1$ and $j\leq n-\overbar{c}$, and  an $1$'s in the positions $(i, j)$ with
$i\geq c+2$ and $j\geq n-\overbar{c}+1$ could cause the  inequality in (\ref{eq5.4}) strict.
Hence if $\overbar{v}=n-c-1$, then $a_{c+1,n-\overbar{c}}=1$ and the equality in (\ref{eq5.4}) holds by the shape of $A\in \mathscr{S}^*(n)$.
\medskip

To prove the second statement, assume $v\geq n-\overbar{c}$ without loss of generality. By (\ref{eq5.1}) and using $s+\overbar{s}\geq 2$, we have
$$\frac{3}{4}(c+\overbar{c})^2 -(n+1)(c+\overbar{c})+n+2\leq -\frac{1}{4}(c-\overbar{c})^2\leq 0.$$
Hence
\begin{align*}
c+\overbar{c}\leq &\frac{2}{3} \left(n+1+\sqrt{(n+1)^2-3(n+2)}\right)\\
                 = &\frac{2}{3} \left(n+1+\sqrt{\left(n-\frac{1}{2}\right)^2-\frac{21}{4}}\right)<\frac{4n}{3}+\frac{1}{3}.
\end{align*}
\end{proof}

\begin{rem}
From \eqref{eq5.2}, we have $\rho(G)+\rho\left(\overbar{G}\right)\leq \frac{4n}{3}+\frac{1}{3}$ for all positive integers $n$, which proves the conjecture of Nikiforov \cite{n:07} that $\rho(G)+\rho\left(\overbar{G}\right)\leq \frac{4n}{3}+O(1)$.
\end{rem}

\begin{cor}\label{cor5.3}
If $c+\overbar{c}\geq \rho_0$,  then $c+\overbar{c}\geq n$, and
either $c+\overbar{c}=\lfloor \frac{4n}{3} \rfloor$ or $c+\overbar{c}=\lfloor \frac{4n}{3} \rfloor-1.$
\end{cor}

\begin{proof} Note that $c+\overbar{c}\geq \rho_0\geq\frac{4n}{3}-\frac{5}{3}$. Then $c+\overbar{c}\geq \lceil\frac{4n}{3}-\frac{5}{3}\rceil\geq n$ for $n\geq 3$.
By (\ref{eq4.4}), the assumption and (\ref{eq5.2}), we have
$$\frac{4n}{3}-\frac{5}{3}\leq  c+\overbar{c} <\frac{4n}{3}+\frac{1}{3}.$$
Since $c+\overbar{c}$ is an integer, either $c+\overbar{c}=\lfloor \frac{4n}{3} \rfloor$ or $c+\overbar{c}=\lfloor \frac{4n}{3} \rfloor-1.$
\end{proof}

Since $c+\overbar{c}$ is an integer, there are the following six cases:
$(n, c+\overbar{c})=(3k, 4k-1)$, $(3k, 4k)$, $(3k+1, 4k)$, $(3k+1, 4k+1)$, $(3k+2, 4k+1)$, or  $(3k+2, 4k+2).$ The following proposition proves Conjecture~\ref{conj} in the special case $(n, c+\overbar{c})=(3k+2, 4k+1)$.

\begin{prop}\label{prop5.4}  If $A\in \mathscr{S}_s^*(n)$, $\rho(A)+\rho(\overbar{A})\geq \rho_0$,  $v\geq n-\overbar{c}$,  and $(n, c+\overbar{c})=(3k+2, 4k+1)$, then $A=A(K_{2k+2}+ N_{k})$ or
$A=A(K_{2k+1}+ N_{k+1})$.
\end{prop}

\begin{proof}
Applying (\ref{eq4.4}), Lemma~\ref{lem3.1}(i),(iii) and the assumptions,
$$4k+1=\frac{4n-5}{3}\leq \rho_0\leq \rho(A)+\rho(\overbar{A})\leq \phi(A)+\phi(\overbar{A})\leq c+\overbar{c}=4k+1.$$
Hence $\rho(A)= \phi(A)=c$ and $\rho(\overbar{A})=\phi(\overbar{A})=\overbar{c}$, which implies $A=A(K_{c+1}+ N_{n-c-1})$ by Lemma~\ref{lem4.1}. It remains to determine the value $c+1$.
{Note that $v=c$, $\overbar{v}=n-c-1$ and $s=2c-v$, $\overbar{s}=2\overbar{c}-\overbar{v}$. Then $s+\overbar{s}=5k+1$. By \eqref{eq5.1}, $5k+1\leq 5k+\frac{5}{4}-\frac{1}{4}(c-\overbar{c})^2$. Then $(c-\overbar{c})^2\leq 1$, and $(c,\overbar{c})=(2k,2k+1)$ or $(c,\overbar{c})=(2k+1,2k)$. Hence $A=A(K_{2k+2}+ N_{k})$ or
$A=A(K_{2k+1}+ N_{k+1})$.
}
\end{proof}

\section{The upper bound $\phi(A)+\phi(\overbar{A})$}\label{s6}

Let $A=(a_{i,j})\in \mathscr{S}^*(n)$ with parameters $c=c(A)$, $v=v(A)$, $s=s(A)$, $\overbar{c}=c(\overbar{A}), \overbar{v}=v(\overbar{A}), \overbar{s}=s(\overbar{A})$.
We will study the upper bound
\begin{equation}\label{eq6.1}
\phi(A)+\phi(\overbar{A})
=\frac{1}{2}\left(v+\overbar{v}-2+\sqrt{E}+\sqrt{F}\right)
\end{equation}
 of $\rho(A)+\rho(\overbar{A})$ described in Lemma~\ref{lem3.1}(i),
 where
 \begin{equation}\label{eq6.2}
 E=(2c-v-1)^2+4s,\quad  F=(2\overbar{c}-\overbar{v}-1)^2+4\overbar{s}.
 \end{equation}

Since the case $(n, c+\overbar{c})=(3k+2, 4k+1)$ has been treated in Proposition~\ref{prop5.4},
we only consider the following $5$  cases
$(n, c+\overbar{c})=(3k, 4k-1)$, $(3k, 4k)$, $(3k+1, 4k)$, $(3k+1, 4k+1)$ or  $(3k+2, 4k+2),$
where $k\in \mathbb{N}$. In the end of this section, Conjecture~\ref{conj} is essentially proved except for the case $(n, c, \overbar{c})=(3k+2, 2k+1, 2k+1)$.  To reduce the complexity of computation, we shall assume $v=n-\overbar{c}$ and $\overbar{v}=n-c-1$ according to Lemma \ref{lem3.3} and Lemma \ref{lem3.4}. The details will be given in the next section. To do this, we need to exclude a special case $2\overbar{c}-\overbar{s}<n-c-1$. In this case $\overbar{v}$ can not be $n-c-1$. By Lemma \ref{lem3.3}, we might assume $\overbar{v}=2\overbar{c}-\overbar{s}$ in this case.
\begin{prop}\label{prop6.1}
If $A\in\mathscr{S}^*$, $(n,c+\overbar{c})\ne (3k+2,4k+1)$, $2\overbar{c}-\overbar{s}<n-c-1$ and $\overbar{v}=2\overbar{c}-\overbar{s}$, then $\phi(A)+\phi(\overbar{A})<\rho_0$.
\end{prop}
\begin{proof}
Note that $\overbar{v}=2\overbar{c}-\overbar{s}<n-c-1$ implies $v\geq n-\overbar{c}$.
Suppose $\phi(A)+\phi(\overbar{A})\geq\rho_0$. Then $c+\overbar{c}\geq n$ by Corollary \ref{cor5.3}.
We first find an upper bound of $s$. From $2\overbar{c}-s<n-c-1$, we have $$\overbar{s}>c+2\overbar{c}-n+1=\displaystyle \frac{3}{2}(c+\overbar{c})-\frac{1}{2}(c-\overbar{c})-n+1.$$
By \eqref{eq5.1},

\begin{eqnarray}
s+\overbar{s}&\leq& \displaystyle-\frac{3}{4}(c+\overbar{c})^2+(n+1)(c+\overbar{c})-\frac{1}{4}(c-\overbar{c})^2-v-\overbar{c}\nonumber \medskip\\
&\leq& \displaystyle-\frac{3}{4}(c+\overbar{c})^2+(n+1)(c+\overbar{c})-\frac{1}{4}(c-\overbar{c})^2-n\nonumber \medskip\\
&=& \left\{\begin{array}{ll}
\displaystyle 4k-\frac{7}{4}-\frac{1}{4}(c-\overbar{c})^2,&\text{if $(n,c+\overbar{c})=(3k,4k-1)$};\medskip\\
\displaystyle k-\frac{1}{4}(c-\overbar{c})^2,&\text{if $(n,c+\overbar{c})=(3k,4k)$};\medskip\\
\displaystyle 5k-1-\frac{1}{4}(c-\overbar{c})^2,&\text{if $(n,c+\overbar{c})=(3k+1,4k)$};\medskip\\
\displaystyle 2k+\frac{1}{4}-\frac{1}{4}(c-\overbar{c})^2,&\text{if $(n,c+\overbar{c})=(3k+1,4k+1)$};\medskip\\
\displaystyle 3k+1-\frac{1}{4}(c-\overbar{c})^2,&\text{if $(n,c+\overbar{c})=(3k+2,4k+2)$}.
\end{array}\right. \label{eq6.3}
\end{eqnarray}
Then
$$s<\left\{\begin{array}{ll}
\displaystyle k-\frac{5}{4}-\frac{1}{4}(c-\overbar{c})^2+\frac{1}{2}(c-\overbar{c}),& \text{if $(n,c+\overbar{c})=(3k,4k-1)$}; \medskip\\
\displaystyle -2k-1-\frac{1}{4}(c-\overbar{c})^2+\frac{1}{2}(c-\overbar{c}),& \text{if $(n,c+\overbar{c})=(3k,4k)$}; \medskip\\
\displaystyle 2k-1-\frac{1}{4}(c-\overbar{c})^2+\frac{1}{2}(c-\overbar{c}),& \text{if $(n,c+\overbar{c})=(3k+1,4k)$}; \medskip\\
\displaystyle -k-\frac{5}{4}-\frac{1}{4}(c-\overbar{c})^2+\frac{1}{2}(c-\overbar{c}),& \text{if $(n,c+\overbar{c})=(3k+1,4k+1)$}; \medskip\\
\displaystyle -1-\frac{1}{4}(c-\overbar{c})^2+\frac{1}{2}(c-\overbar{c}),& \text{if $(n,c+\overbar{c})=(3k+2,4k+2)$}.
\end{array}\right.$$
Since $-\frac{1}{4}(c-\overbar{c})^2+\frac{1}{2}(c-\overbar{c})\leq\frac{1}{4}$ and $s>0$, we need only to consider the cases $(n,c+\overbar{c})=(3k,4k-1)$ and $(n,c+\overbar{c})=(3k+1,4k)$.

From $\overbar{v}<n-c-1$ and $c+\overbar{c}\geq n$, we have $\overbar{v}<\overbar{c}-1<\overbar{c}$. By Lemma \ref{lem3.5}, we might assume $v=\max\{2c-s,n-\overbar{c}\}$. We shell show that $v=2c-s$ is impossible. Suppose $v=2c-s$. From \eqref{eq5.1}, we have
$$2c-v+2\overbar{c}-\overbar{v}=s+\overbar{s}\leq -\frac{3}{4}(c+\overbar{c})^2+(n+1)(c+\overbar{c})-\frac{1}{4}(c-\overbar{c})^2-v-\overbar{c}.$$
Then
$$c+\overbar{c}\leq\frac{2}{3}\left(n-1+\sqrt{(n-1)^2-3\left[\frac{1}{4}(c-\overbar{c})^2+(\overbar{c}-\overbar{v})\right]}\right)<\frac{4n-4}{3},$$
contradict to $(n,c+\overbar{}c)=(3k,4k-1)$ or $(3k+1,4k)$. Hence $v=n-\overbar{c}$.
Therefore, $\phi(\overbar{A})=\overbar{c}$ by Lemma \ref{lem3.1}(iii), and
$$\begin{array}{lll}
\phi(A)+\phi(\overbar{A})
&=&\frac{1}{2}\left(n-\overbar{c}-1+\sqrt{(2c-(n-\overbar{c})-1)^2+4s}\right)+\overbar{c}\\
&=&\frac{1}{2}\left(n+\overbar{c}-1+\sqrt{(2c-(n-\overbar{c})-1)^2+4s}\right).
\end{array}$$
Note that
$$\begin{array}{lll}
\rho_0-(\phi(A)+\phi(\overbar{A}))&\geq& \displaystyle\frac{4n-5}{3}-(\phi(A)+\phi(\overbar{A}))\medskip\\
&=&\displaystyle\frac{1}{2}\left(\frac{5n-7}{3}-\overbar{c}-\sqrt{(2c+\overbar{c}-n-1)^2+4s}\right).
\end{array}$$
It suffices to show
$$\left(\frac{5n-7}{3}-\overbar{c}\right)^2>(2c+\overbar{c}-n-1)^2+4s.$$
Note that $\overbar{c}=-(c+\overbar{c})/2+(c-\overbar{c})/2$ and $2c+\overbar{c}=3(c+\overbar{c})/2+(c-\overbar{c})/2$. Then
\begin{align*}
&\left(\frac{5n-7}{3}-\overbar{c}\right)^2-((2c+\overbar{c}-n-1)^2+4s)\\
>&\left\{\begin{array}{ll}
\displaystyle(c-\overbar{c})^2-\frac{4}{3}(c-\overbar{c})+\frac{19}{9},& \text{if $(n,c+\overbar{c})=(3k,4k-1)$}\medskip\\
\displaystyle(c-\overbar{c})^2-\frac{2}{3}(c-\overbar{c})+\frac{4}{9},& \text{if $(n,c+\overbar{c})=(3k+1,4k)$}
\end{array}\right.\\
>&0,
\end{align*}
completing the proof.
\end{proof}

\begin{lem}\label{lem6.2}
Referring to the notations in (\ref{eq6.1})-(\ref{eq6.2}), $\phi(A)+\phi(\overbar{A})$ is a zero of the following  quartic polynomial
$$g(x)=(2x-v-\overbar{v}+2)^2((2x-v-\overbar{v}+2)^2-2(E+F))+(E-F)^2.$$
\end{lem}
\begin{proof}
This is a direct computation from \eqref{eq6.1}.
\end{proof}

We will prove that $g(x)$ is increasing on $[(4n-5)/3, \infty)$, and $g(\rho_0)\geq0$ except for the case $(n, c, \overbar{c})=(3k+2, 2k+1, 2k+1)$ under the assumption $v= n-\overbar{c}$ and $\overbar{v}=n-c-1$.

\begin{lem}\label{lem6.3} If $c+\overbar{c}\geq(4n-5)/3$, $v= n-\overbar{c}$ and $\overbar{v}=n-c-1$, then
\begin{align}
&E+F \nonumber \\=&\frac{3}{2}(c+\overbar{c})^2-(2n-1)(c+\overbar{c})+2n^2-2n+1 -\frac{1}{2}(c-\overbar{c})^2-(c-\overbar{c}). \label{eq6.4}
\end{align}
Moreover if  $x\geq ({4n}-{5})/{3}$,  then
$(2x-v-\overbar{v}+2)^2-(E+F)>0.$
\end{lem}
\begin{proof}
By (\ref{eq6.2}), the equality in (\ref{eq5.1}) and the assumptions, we have
\begin{align*}
E+F=& (2c-(n-\overbar{c})-1)^2+4s+(2\overbar{c}-(n-1-c)-1)^2+4\overbar{s}\\
   =&(2c+\overbar{c}-n-1)^2+(c+2\overbar{c}-n)^2\\
    &\quad +4\left(-\frac{3}{4}(c+\overbar{c})^2+(n+1)(c+\overbar{c})-\frac{1}{4}(c-\overbar{c})^2-v-\overbar{c}\right).
   \end{align*}
Simplifying the above equation, we immediately have (\ref{eq6.4}).
If  $x\geq ({4n}-{5})/{3}$, then by (\ref{eq6.4}), $n\geq 3$ and $(c-\overbar{c})^2+2(c-\overbar{c})\geq -1$,
\begin{align*}
&(2x-v-\overbar{v}+2)^2-(E+F)\\
\geq&  \left(2\left(\frac{4n-5}{3}\right)-(n-\overbar{c})-(n-1-c)+2\right)^2\\
&-\left(\frac{3}{2}(c+\overbar{c})^2-(2n-1)(c+\overbar{c})+2n^2-2n+1 -\frac{1}{2}(c-\overbar{c})^2-(c-\overbar{c})\right)\\
=&-\frac{1}{2}(c+\overbar{c})^2+ \frac{10n-5}{3}(c+\overbar{c})+\frac{-14n^2+14n-8}{9}+\frac{1}{2}(c-\overbar{c})^2+(c-\overbar{c})\\
\geq& -\frac{1}{2}\left(\frac{4n-5}{3}\right)^2+ \frac{10n-5}{3}\left(\frac{4n-5}{3}\right)+\frac{-14n^2+14n-8}{9}+\frac{1}{2}(c-\overbar{c})^2+(c-\overbar{c})\\
\geq&\frac{36n^2-72n+9}{18}-\frac{1}{2}=2n(n-2)>0.
\end{align*}
\end{proof}

\begin{prop}\label{prop6.4}  If $A\in \mathscr{S}^*(n)$, $c+\overbar c\geq (4n-5)/3$,  $v=n-\overbar{c}$ and $\overbar{v}=n-c-1$,  then the function $g(x)$  is increasing on the interval $[(4n-5)/3, \infty)$.
\end{prop}

\begin{proof} The derivative of the $g(x)$ in Lemma~\ref{lem6.2} is
$$g'(x)=8(2x-v-\overbar{v}+2)((2x-v-\overbar{v}+2)^2-E-F).$$
It suffices to show  $g'(x)>0$ for $x\in [{4n-5}/{3}, \infty).$
This follows since
$$2x-v-\overbar{v}+2\geq 2\left( \frac{4n-5}{3}\right)-(2n-c-\overbar{c}-1)+2 \geq 2n-2> 0,$$
and since $(2x-v_1-v_2+2)^2-E-F>0$ by Lemma~\ref{lem6.3}.
\end{proof}

\begin{prop}\label{prop6.5} If $A\in \mathscr{S}^*(n)$, $\rho(A)+\rho(\overbar{A})\geq \rho_0$,  $v= n-\overbar{c}$, $\overbar{v}=n-c-1$, and $(n,c+\overbar{c})\ne(3k+2,4k+1)$, $(n, c, \overbar{c})\not=(3k+2, 2k+1, 2k+1)$, then $g(\rho_0)\geq 0$.
Moreover $g(\rho_0)= 0$ if and only if $(n,c,\overbar{c})=(3k,2k-1,2k)$ or $(3k+1,2k,2k)$.
\end{prop}

\begin{proof} Considering the three cases $n=3k$, $3k+1$ or $3k+2$ in  (\ref{eq4.3}) separately, we have
\begin{align}\rho_0=&\left\{
          \begin{array}{ll}
            \frac{1}{2}(5k-3+\sqrt{9k^2-2k+1}), & \hbox{if $n=3k$;} \\
            \frac{1}{2}(5k-1+\sqrt{9k^2+2k+1}), & \hbox{if $n=3k+1$;} \\
            4k+1, & \hbox{if $n=3k+2$}
          \end{array}
        \right. \label{eq6.5} \\
=& \frac{4n-5}{3}+u_n \nonumber \\
=&\left\{
          \begin{array}{ll}
            \frac{1}{3}(12k-5)+u_n, & \hbox{if $n=3k$;} \\
            \frac{1}{3}(12k-1)+u_n, & \hbox{if $n=3k+1$;} \\
            4k+1, & \hbox{if $n=3k+2$,}
          \end{array}
        \right. \label{eq6.6}
\end{align}
where
\begin{align*}
u_n:=&\rho_0-\frac{4n-5}{3}\\
=&\left\{
          \begin{array}{ll}
            \frac{1}{2}(5k-3+\sqrt{9k^2-2k+1})-\frac{12k-5}{3}, & \hbox{if $n=3k$;} \\
            \frac{1}{2}(5k-1+\sqrt{9k^2+2k+1})-\frac{12k-1}{3}, & \hbox{if $n=3k+1$;} \\
            4k+1-(4k+1), & \hbox{if $n=3k+2$.}
          \end{array}
        \right.\\
    =&\displaystyle\left\{
        \begin{array}{ll}
          \frac{4}{3\left(9k-1+\sqrt{(9k-1)^2+8}\right)}, & \hbox{if $n=3k$;} \\
          \frac{4}{3\left(9k+1+\sqrt{(9k+1)^2+8} \right)}, & \hbox{if $n=3k+1$;} \\
          0, & \hbox{if $n=3k+2$,}
        \end{array}
      \right.
\end{align*}
Hence
\begin{equation}\label{eq6.7}
\left\{
  \begin{array}{ll}
    \frac{4}{54k-3}<u_n<\frac{2}{27k-3}, & \hbox{if $n=3k$;} \\
    \frac{4}{54k+9}<u_n<\frac{2}{27k+3}, & \hbox{if $n=3k+1$;}\\
    u_n=0, & \hbox{if $n=3k+2$.}
\end{array}
\right.
\end{equation}
From (\ref{eq6.5}),  $\rho_0$ is a zero of
\begin{equation}\label{eq6.8}
f(x)=\left\{
       \begin{array}{ll}
         x^2-(5k-3)x+4k^2-7k+2, & \hbox{if $n=3k$;} \\
         x^2-(5k-1)x+4k^2-3k, & \hbox{if $n=3k+1$;} \\
         x^2-(5k+1)x+4k^2+k, & \hbox{if $n=3k+2$.}
       \end{array}
     \right.
\end{equation}
We will use (\ref{eq6.5})-(\ref{eq6.8}) to compute $g(\rho_0)$.
By using the assumptions, $f(\rho_0)=0$ in (\ref{eq6.8}) and the $\rho_0$ in (\ref{eq6.6}),  we have
\begin{align}
&(2\rho_0-v-\overbar{v}+2)^2=(2\rho_0-(n-\overbar{c})-(n-c-1)+2)^2-4f(\rho_0) \nonumber\\
=&4(\rho_0^2-f(\rho_0))+4(-2n+c+\overbar{c}+3)\rho_0+(-2n+c+\overbar{c}+3)^2\nonumber\\
=&\left\{
   \begin{array}{ll}
  36k^2-16k+\frac{8}{3}+(12k-4)u_n, & \hbox{if $(n, c+\overbar{c})=(3k, 4k-1)$;} \\
  36k^2-4k+1+12ku_n, & \hbox{if $(n, c+\overbar{c})=(3k, 4k)$;} \\
  36k^2+4k+1+12ku_n, & \hbox{if  $(n, c+\overbar{c})=(3k+1, 4k)$ ;} \\
  36k^2+16k+\frac{8}{3}+(12k+4)u_n    , & \hbox{if $(n, c+\overbar{c})=(3k+1, 4k+1)$;} \\
  36k^2+36k+9, & \hbox{if $(n, c+\overbar{c})=(3k+2, 4k+2)$.}
   \end{array}
 \right. \label{eq6.9}
\end{align}
By (\ref{eq6.4}), we have
\begin{equation}
E+F =\left\{
   \begin{array}{ll}
   18k^2-8k+\frac{3}{2}-\frac{T}{2}, & \hbox{if $(n, c+\overbar{c})=(3k, 4k-1)$;} \smallskip\\
   18k^2-2k+1-\frac{T}{2} , & \hbox{if $(n, c+\overbar{c})=(3k, 4k)$;} \smallskip\\
18k^2+2k+1-\frac{T}{2}, & \hbox{if  $(n, c+\overbar{c})=(3k+1, 4k)$ ;} \smallskip\\
18k^2+8k+\frac{3}{2}-\frac{T}{2}    , & \hbox{if $(n, c+\overbar{c})=(3k+1, 4k+1)$;} \smallskip\\
18k^2+18k+5-\frac{T}{2}  , & \hbox{if $(n, c+\overbar{c})=(3k+2, 4k+2)$,}
   \end{array}
 \right. \label{eq6.10}
\end{equation}
where $T=(c-\overbar{c})^2+2(c-\overbar{c}).$
By (\ref{eq6.9})-(\ref{eq6.10}), we have
\begin{align}
&(2\rho_0-v-\overbar{v}+2)^2-2(E+F)\nonumber \\
= &\left\{
   \begin{array}{ll}
   -\frac{1}{3}+\left( 12k-4\right)u_n+T, & \hbox{if $(n, c+\overbar{c})=(3k, 4k-1)$;} \\
    -1+12ku_n+T , & \hbox{if $(n, c+\overbar{c})=(3k, 4k)$;} \\
 -1+12ku_n+T , & \hbox{if  $(n, c+\overbar{c})=(3k+1, 4k)$ ;} \\
 -\frac{1}{3}+(12k+4)u_n+T     , & \hbox{if $(n, c+\overbar{c})=(3k+1, 4k+1)$;} \\
 -1+T   , & \hbox{if $(n, c+\overbar{c})=(3k+2, 4k+2)$,}
   \end{array}
 \right. \label{eq6.11}\\
\geq & -1+T.\nonumber
\end{align}
 If $c-\overbar{c}\not\in \{0, -1, -2\}$ then $T=(c-\overbar{c})^2+2(c-\overbar{c})\geq 3$ and $g(\rho_0)>0$ with referring to the expression of $g(x)$ in Lemma~\ref{lem6.2}. We investigate the remaining cases according to list of the above five cases of $(n, c+\overbar{c})$ under the assumption that $c-\overbar{c}\in \{0, -1, -2\}$ and $(n, c, \overbar{c})\not=(3k+2, 2k+1, 2k+1)$. In the following case by case discussions, we use  the inequality $0<s\leq 2c-v$ by Lemma \ref{lem3.1}(ii),
the equation
$$s+\overbar{s}=-\frac{3}{4}(c+\overbar{c})^2+(n+1)(c+\overbar{c})-\frac{1}{4}(c-\overbar{c})^2-n$$
from the equality case of (\ref{eq5.1}) and its value in \eqref{eq6.3}, and the formula of $E$ and $F$ in (\ref{eq6.2}) many times without further notice.
\medskip

\noindent  {\bf Case 1.} $n=3k$ and $(c,\overbar{c}) = (2k-1,2k)$: Then $(v,\overbar{v})=(k,k)$, $s+\overbar{s}=4k-2$, $E=(3k-3)^2+4s$,  $F=(3k-1)^2+4\overbar{s}$, and $s\leq 3k-2$.  Hence $\overbar{s}\leq k$, $s-\overbar{s}\leq (3k-2)-k= 2k-2$ and $E\leq (3k-1)^2<F$. By Lemma~\ref{lem2.8}, $\sqrt{E}+\sqrt{F}$ takes maximum at $(s, \overbar{s})=(3k-2, k).$ By (\ref{eq6.1}) and  (\ref{eq6.5}),
\begin{align*}
\rho_0\leq &\rho(A)+\rho(\overbar{A})\leq \phi(A)+\phi(\overbar{A})=\frac{1}{2}\left(v+\overbar{v}-2+\sqrt{E}+\sqrt{F}\right)\\
\leq &\frac{1}{2}\left(2k-2+\sqrt{(3k-3)^2+4(3k-2)}+\sqrt{(3k-1)^2+4k}\right)=\rho_0.
\end{align*}
Hence $\rho_0=\rho(A)+\rho(\overbar{A})=\phi(A)+\phi(\overbar{A})$ and $g(\rho_0)=0$.
\medskip

\noindent{\bf Case 2.} $n=3k$ and $(c,\overbar{c})\in \{(2k,2k),(2k-1,2k+1)\}$:
First, suppose $(c,\overbar{c})=(2k,2k)$. Then $(v,\overbar{v})=(k,k-1)$, $s+\overbar{s}=k$,
$E=(3k-1)^2+4s$ and  $F=(3k)^2+4\overbar{s}$. Hence $s-\overbar{s}\leq (k-1)-1=k-2$, and
$$E-F=(3k-1)^2-(3k)^2+4(s-\overbar{s})\leq -2k-7.$$
Next suppose $(c,\overbar{c})=(2k-1,2k+1)$. Then $s+\overbar{s}=k-1$, $(v,\overbar{v})=(k-1,k)$,
$E=(3k-2)^2+4s$, $F=(3k+1)^2+4\overbar{s}$. Hence $s-\overbar{s}\leq k-3$, and
$$E-F=(3k-2)^2-(3k+1)^2+4(s-\overbar{s})\leq -14k-9.$$
Thus  $(E-F)^2\geq (2k+7)^2$ for $(c,\overbar{c})\in \{(2k,2k),(2k-1,2k+1)\}$.
By (\ref{eq6.9}), (\ref{eq6.11}), and (\ref{eq6.7}), we have $(2\rho_0-v-\overbar{v}+2)^2\leq 36k^2-4k+2$ and
\begin{align*}(2\rho_0-v-\overbar{v}+2)^2-2(E+F)=&-1+12u_n+T
>-1+48/(54k-3)+0\\>&-1+8/9= -1/9.\end{align*}
Hence
\begin{align*}
g(\rho_0)=&(2\rho_0-v-\overbar{v}+2)^2((2\rho_0-v-\overbar{v}+2)^2-2(E+F))+(E-F)^2 \\
>&(36k^2-4k+2)\left(-\frac{1}{9}\right)+(2k+7)^2>0.
\end{align*}

\noindent{\bf Case 3.} $n=3k+1$ and $(c,\overbar{c})\in  \{(2k,2k),(2k-1,2k+1)\}$:
First suppose $(c,\overbar{c})=(2k,2k)$. Then $(v,\overbar{v})=(k+1,k)$, $s+\overbar{s}=5k-1$, $E=(3k-2)^2+4s$,  $F=(3k-1)^2+4\overbar{s}$, and $s\leq 3k-1$. Then $\overbar{s}\geq 2k$ and $E\leq (3k)^2<(3k-1)^2+4\cdot 2k\leq F$. Thus by Lemma~\ref{lem2.8}, $\sqrt{E}+\sqrt{F}$ takes maximum at $(s,\overbar{s})=(3k-1,2k)$.
Therefore by (\ref{eq6.1}) and (\ref{eq6.5}),
\begin{align*}
\rho_0\leq&\rho(A)+\rho\left(\overbar{A}\right)\leq \phi(A)+\phi\left(\overbar{A}\right)=\frac{1}{2}\left(v+\overbar{v}-2+\sqrt{E}+\sqrt{F}\right) \\
      \leq& \frac{1}{2}\left(2k-1+\sqrt{(3k-2)^2+4(3k-1)}+\sqrt{(3k-1)^2+4\cdot 2k}\right) \\
         =& \frac{1}{2}\left(5k-1+\sqrt{9k^2+2k+1}\right)=\rho_0.
\end{align*}
As in the case 1, we have $g(\rho_0)=0$.
\medskip

Next suppose $(c,\overbar{c})=(2k-1,2k+1)$. Then $(v, \overbar{v})=(k,k+1)$, $s+\overbar{s}=5k-2$, $E=(3k-3)^2+4s$,  $F=(3k)^2+4\overbar{s}$, and  $s\leq 3k-2$.  Hence  $E\leq(3k-1)^2<F$. Then by Lemma~\ref{lem2.8}, $\sqrt{E}+\sqrt{F}$ takes maximum at $(s,\overbar{s})=(3k-2,2k)$.
Therefore  by (\ref{eq6.1}) and (\ref{eq6.5}),
\begin{align*}
 \rho_0  \leq  &\rho(G)+\rho\left(\overbar{G}\right)\leq \phi(A)+\phi\left(\overbar{A}\right)=\frac{1}{2}\left(v+\overbar{v}-2+\sqrt{E}+\sqrt{F}\right) \\
\leq& \frac{1}{2}\left(2k-1+\sqrt{(3k-3)^2+4(3k-2)}+\sqrt{(3k-1)^2+4\cdot 2k}\right) \\
=& \frac{1}{2}\left(5k-2+\sqrt{9k^2+2k+1}\right)<\rho_0,
\end{align*}
a contradiction.
\medskip

\noindent{\bf Case 4.} $n=3k+1$ and $(c, \overbar{c})=(2k, 2k+1)$:
 Then $(v,\overbar{v})=(k,k)$, $s+\overbar{s}=2k$, $E=(3k-1)^2+4s$, and  $F=(3k+1)^2+4\overbar{s}$. Hence $s-\overbar{s}\leq 2k-2$, and  $$E-F=(3k-1)^2-(3k+1)^2+4(s-\overbar{s})\leq -12k+4(2k-2)=-4k-8.$$
By (\ref{eq6.9}), (\ref{eq6.11}), and (\ref{eq6.7}), we have $(2\rho_0-v-\overbar{v}+2)^2\leq 36k^2+16k+4$ and
\begin{align*}
(2\rho_0-v-\overbar{v}+2)^2-2(E+F)=&-1/3+(12k+4)u_n+T\\
>&-1/3+(48k+16)/(54k+9)-1\\
>&-1/3+8/9-1= -4/9.
\end{align*}
Hence
\begin{align*}
g(\rho_0)
=&(2\rho_0-v-\overbar{v}+2)^2((2\rho_0-v-\overbar{v}+2)^2-2(E+F))+(E-F)^2 \\
>&(36k^2+16k+4)\left(-\frac{4}{9}\right)+(4k+8)^2= 64\left(\frac{8}{9}k+\frac{35}{36}\right)>0.
\end{align*}

\noindent{\bf Case 5.} $n=3k+2$ and $(c, \overbar{c})=(2k, 2k+2)$: Then $(v,\overbar{v})=(k,k+1)$, $s+\overbar{s}=3k$, $E=(3k-1)^2+4s$ and $F=(3k+2)^2+4\overbar{s}$. Hence $E<(3k+1)^2<F$. By Lemma~\ref{lem2.8}, we have $\sqrt{E}+\sqrt{F}<\sqrt{(3k-1)^2+4\cdot 3k}+\sqrt{(3k+2)^2+4\cdot 0}=6k+3$ . Therefore,
\begin{align*}
\rho_0\leq &\rho(A)+\rho\left(\overbar{A}\right) \leq \phi(A)+\phi\left(\overbar{A}\right)=\frac{1}{2}\left(v+\overbar{v}-2+\sqrt{E}+\sqrt{F}\right) \\
<& \frac{1}{2}(2k-1+6k+3) = 4k+1=\rho_0,
\end{align*}
a contradiction.
\end{proof}

\section{Proof of Conjecture \ref{conj}}\label{s7}

Let $G$ be a graph of order $n\geq 3$, $A=A(G)$ and $\rho(A)+\rho(\overbar{A})\geq \rho_0$, where $\rho_0$ is  defined before Proposition~\ref{prop4.3}. To prove Conjecture \ref{conj}, we shall prove that $G$ or its complement graph $\overbar{G}$ is the complete split graph
$K_{\lfloor\frac{n}{3}\rfloor}\vee N_{\lceil\frac{2n}{3}\rceil}$, and if $n\equiv 2 \ (\text{mod }  3)$,  $K_{\lceil\frac{n}{3}\rceil}\vee N_{\lfloor\frac{2n}{3}\rfloor}$ is also possible.
By Lemma \ref{lem2.1},  $A\in\mathscr{S}^*_s(n)$.
Let $c=c(A)$, $v=v(A)$, $s=s(A)$, $\overbar{c}=c(\overbar{A})$, $\overbar{v}=v(\overbar{A})$ and $\overbar{s}=s(\overbar{A})$ denote the parameters of $A$. By interchanging $A$ and $\overbar{A}$ if necessary, we might assume $a_{c+1, n-\overbar{c}}=1$, or equivalently $v\geq n-\overbar{c}$ or $\overbar{v}\leq n-c-1.$
By  Corollary~\ref{cor5.3}, we have $c+\overbar{c}=\left\lfloor \frac{4n}{3}\right\rfloor$ or $c+\overbar{c}=\left\lfloor \frac{4n}{3}\right\rfloor-1$.
\medskip

Suppose $(n, c+\overbar{c})=(3k+2, 4k+1)$. Then $A=A(K_{2k+2}+ N_{k})$ or $A=A(K_{2k+1}+ N_{k+1})$ by Proposition~\ref{prop5.4}.
Since the complement graph of $K_{2k+2}+ N_{k}$ is the complete split graph $K_k\vee N_{2k+1}$ with $k=\left\lfloor \frac{n}{3}\right\rfloor$ and the complement graph of  $K_{2k+1}+ N_{k+1}$ is the complete split graph  $K_{k+1}\vee N_{2k+1}$
with $k+1=\left\lceil \frac{n}{3}\right\rceil$,  so the proof is finished in this case.
\medskip

Suppose $(n, c+\overbar{c})\not=(3k+2, 4k+1)$. We need to consider matrices which are not symmetric. First, apply Lemma~\ref{lem3.3} to find $A_1\in \mathscr{S}^*(n)$ with parameters $c_1=c(A_1)=c$, $v_1=v(A_1)=v$, $s_1=s(A_1)=s$, $\overbar{c}_1=\overbar{c}(A_1)=\overbar{c}$, $\overbar{v}_1=\overbar{v}(A_1)=\min\{2\overbar{c}_1-\overbar{s}_1,n-c_1\}\geq \overbar{v}$, $\overbar{s}_1=\overbar{s}(A_1)\geq\overbar{s}$, and $\phi(A)=\phi(A_1)$, $\phi(\overbar{A})\leq \phi(\overbar{A}_1)$.
If $2\overbar{c}_1-\overbar{s}_1< n-c_1-1$, then $\overbar{v}_1=2\overbar{c}_1-\overbar{s}_1$. By Proposition \ref{prop6.1} with its $A$ being $A_1$,
$$\rho_0\leq \rho(A)+\rho(\overbar{A})\leq\phi(A)+\phi(\overbar{A})\leq\phi(A_1)+\phi(\overbar{A}_1)<\rho_0,$$
a contradiction.
Now assume $2\overbar{c}_1-\overbar{s}_1\geq n-c_1-1$. Then $\overbar{v}_1=n-c_1-1$.
Applying Lemma~\ref{lem3.4} with its $A$ being $A_1$ here and its $v'$ being $n-\overbar{c}_1$ here to find $A_2\in \mathscr{S}^*(n)$ with parameters $c_2=c(A_2)$, $v_2=v(A_2)$, $s_2=s(A_2)$, $\overbar{c}_2=\overbar{c}(A_2)$, $\overbar{v}_2=\overbar{v}(A_2)$, $\overbar{s}_2=\overbar{s}(A_2)$ satisfying $c_2=c_1$, $v_2=n-\overbar{c}_2$, $s_2+v_2=s_1+v_1$, $\overbar{c}_2=\overbar{c_1}$, $\overbar{v}_2=\overbar{v}_1=n-c_2-1$, $\overbar{s}_2=\overbar{s}_1$ and $\phi(A_1)\leq \phi(A_2)$, $\phi(\overbar{A}_1)=\phi(\overbar{A}_2)$.
Putting the above inequalities together,
\begin{equation}\label{eq7.1}\rho_0\leq \rho(A)+\rho(\overbar{A})\leq \phi(A)+\phi(\overbar{A})\leq \phi(A_1)+\phi(\overbar{A_1})\leq \phi(A_2)+\phi(\overbar{A}_2).\end{equation}
Notice that the parameters $c_2, v_2, s_2, \overbar{c}_2, \overbar{v}_2, \overbar{s}_2$
of $A_2$ satisfy the assumptions $c_2=c$, $\overbar{c}_2=\overbar{c}$, $v_2=n-\overbar{c}_2$ and $\overbar{v}_2=n-c_2-1$.

\medskip

Further suppose $(n, c, \overbar{c})\not=(3k+2, 2k+1, 2k+1)$. Let $g(x)$ be the polynomial
with $g\left(\phi(A_2)+\phi(\overbar{A}_2)\right)=0$, defined in Lemma~\ref{lem6.2}. Then $g(\rho_0)\geq 0$ by Proposition~\ref{prop6.5}.
Since $g(x)$ is increasing on $[(4n-5)/3, \infty)$ by Proposition~\ref{prop6.4}, and $\rho_0\in [(4n-5)/3, \infty),$ we have
$\phi(A_2)+\phi(\overbar{A}_2)\leq \rho_0$, which implies
$\rho(A)+\rho\left(\overbar{A}\right)=\phi(A)+\phi(\overbar{A})=\phi(A_2)+\phi(\overbar{A}_2)=\rho_0$ by (\ref{eq7.1}),
and $g(\rho_0)=0$. Then $(n,c,\overbar{c})=(n,c_2,\overbar{c}_2)=(3k,2k-1,2k)$ or $(3k+1,2k,2k)$ by the second part of Proposition \ref{prop6.5}, which implies
$A=A(K_{c+1}+N_{n-c-1})=A(K_{n-\lfloor\frac{n}{3}\rfloor}+ N_{\lfloor\frac{n}{3}\rfloor})$ by Lemma~\ref{lem4.1}(i).
Since the complement graph of $K_{n-\lfloor\frac{n}{3}\rfloor}+ N_{\lfloor\frac{n}{3}\rfloor}$ is the complete split graph $K_{\lfloor\frac{n}{3}\rfloor}\vee N_{n-\lfloor\frac{n}{3}\rfloor}$, the proof is finished.
\medskip

Now consider the last case. Assume $(n, c, \overbar{c})=(3k+2, 2k+1, 2k+1)$. In this case $\rho_0=4k+1$ by (\ref{eq4.3}). We will show that this case
is impossible.
We turn back to consider the parameters of $A$ rather than $A_2$.
Recall that $v\geq n-\overbar{c}=k+1$ and $\overbar{v}\leq n-c-1=k$.
By (\ref{eq5.1}), $s+\overbar{s}\leq4k+2-v=2c-v$. Then $s<2c-v$ since $\overbar{s}>0$. Hence
$E=(2c-v-1)^2+4s< (2c-v+1)^2$ and $F=(2\overbar{c}-\overbar{v}-1)^2+4\overbar{s}$ by (\ref{eq6.2}).
If $v-\overbar{v}\geq 2$, then  $E<  (2c-v+1)^2\leq (2\overbar{c}-\overbar{v}-1)^2<F$. By (\ref{eq6.1}) and Lemma~\ref{lem2.8}, $\sqrt{E}+\sqrt{F}$ is less than the value taken at $(s,\overbar{s})=(2c-v,0)$. Then
\begin{align*}
&\rho(G)+\rho\left(\overbar{G}\right)\leq \phi(A)+\phi\left(\overbar{A}\right)=\frac{1}{2}(v+\overbar{v}-2+\sqrt{E}+\sqrt{F}) \\
<&\frac{1}{2}\left(v+\overbar{v}-2+\sqrt{(2c-v+1)^2}+\sqrt{(2\overbar{c}-\overbar{v}-1)^2}\right)  \\
=&c+\overbar{c}-1=4k+1=\rho_0,
\end{align*}
a contradiction.
\medskip

Next assume $v-\overbar{v}\leq 1$.  Then $v=k+1$, $\overbar{v}=k$ since $v\geq k+1$ and $\overbar{v}\leq k$, and $s+\overbar{s}=3k+1$ by the equality in \eqref{eq5.1} since $v=n-\overbar{c}$. Hence $s\leq 3k$.  We need another upper bound of  $\rho(A)+\rho\left(\overbar{A}\right)$ to deal with this case. Let
$$M_1=\begin{bmatrix}
2k& s \\
1& -k
\end{bmatrix}
\quad\text{and}\quad
M_2=\begin{bmatrix}
k-1 & k & 3k+1-s-a \\
k+1 & k & a \\
1 & 0 & 0
\end{bmatrix},
$$
where $$a=\sum\limits_{i=\overbar{v}+1}\limits^{\overbar{c}}(r_i\left(\overbar{A}\right)-\overbar{c}+1).$$
Note that $A$ is symmetric. From $a_{c+1, v+1}=0$, $a_{v, c+1}=1$ and the shape of $A$, we find $-k(k-1)=-(\overbar{c}-(\overbar{v}+1))(\overbar{c}-(\overbar{v}+1)-1)\leq a \leq 0.$
Note that $$\rho_r(M_1)=\phi(A)$$ by (\ref{eq3.2}). We will show that $\rho\left(\overbar{G}\right)\leq\rho_r(M_2)$. Let $M_3=(m_{ij})$ be a $(\overbar{c}+1)\times(\overbar{c}+1)$ matrix, where
$$m_{ij}=\left\{\begin{array}{ll}
1, &  \text{if $1\leq i,j\leq \overbar{c}$, $i\ne j$, or $i=\overbar{c}+1, 1\leq j\leq \overbar{v}$;}\\
r_i\left(\overbar{A}\right)-\overbar{c}+1, & \text{if $1\leq i\leq \overbar{c}$, $j=\overbar{c}+1$;}\\
0,& \text{otherwise},
\end{array}\right.$$
and let $\Pi_1=\{\{1\},\{2\},\ldots,\{\overbar{c}\},\{\overbar{c}+1,\ldots,n\}\}$. Then $\rho(\overbar{A})\leq \rho_r(M_3)$ by applying Theorem \ref{thm2.5} with $C=\overbar{A}$ and $\Pi=\Pi_1$. Let $\Pi_2=\{\{1,\ldots, \overbar{v}\},\{\overbar{v}+1,\ldots, \overbar{c}\},\{\overbar{c}+1\}\}$. Note that
$\sum_{i=1}^{\overbar{v}}(r_i(\overbar{A})-\overbar{c}+1)=\overbar{s}-1=3k+1-s-a$ and $\Pi_2$ is an equitable partition of $M_3^T$.
Then $M_2^T=\Pi_2(M_3^T)$ and $\rho_r(M_3)=\rho_r(\Pi_2(M_3^T))=\rho_r(M_2)$ by Lemma \ref{lem2.6}.
So $\rho_r(M_1)+\rho_r(M_2)$ is an upper bound of $\rho(A)+\rho\left(\overbar{A}\right)$.
Note that $\rho_r(M_1)+\rho_r(M_2)$ is equal to $\rho_r(M)$ by Lemma \ref{lem2.9}, where $M=M_1\otimes I_3+I_2\otimes M_2$. We will show that $\rho_r(M)<4k+1$, a contradiction to the assumption $\rho(A)+\rho\left(\overbar{A}\right)\geq \rho_0$ in the beginning. Note that
$$M=\begin{bmatrix}
3k-1 & k & 3k+1-s-a & s & 0 & 0 \\
k+1 & 3k & a & 0 & s & 0 \\
1 & 0 & 2k & 0 & 0 & s \\
1 & 0 & 0 & -1 & k & 3k+1-s-a \\
0 & 1 & 0 & k+1 & 0 & a \\
0 & 0 & 1 & 1 & 0 & -k
\end{bmatrix}.$$
Let $\det(xI_6-M)$ be the characteristic polynomial of $M$. Then the sum of the terms divisible by $a$ in $\det((4k+1)I_6-M)$ is
$$\begin{array}{ll}
&(3k+1-s)a^2+(78k^3-26k^2s+65k^2-ks-4s^2+19k+2s+2)a \\
=& \displaystyle \left(a+\frac{26k^2+13k+4s+2}{2}\right)^2-\left(\frac{26k^2+13k+4s+2}{2}\right)^2.
\end{array}$$
Since $$
\displaystyle-\frac{26k^2+13k+4s+2}{2}<-k(k-1)\leq a\leq 0,$$ we have
\begin{align*}
&\det((4k+1)I_6-M)\geq \det((4k+1)I_6-M)|_{a=-k(k-1)} \\
=&\det\left([(4k+1)I_6-M]_{a=-k(k-1)}\right) \\
=&(216k^4-120k^3s+16k^2s^2)+(252k^3-130k^2s+30ks_1^2-4s^3) \\
&+(105k^2-48ks+9s^2)+(18k-6s)+1  \\
=& (3k-s)(72k^3-16k^2s)+\left((3k-s)\left(84k^2-\frac{46}{3}ks+\frac{44}{9}s^2\right)+\frac{8}{9}s^3\right) \\
& +\left((3k-s)\left(35k-\frac{13}{3}s\right)+\frac{14}{3}s^2\right)+6(3k-s)+1 > 0.
\end{align*}
 Note that $\rho_r(M)$ is the largest zero of $\det(xI_6-M)$. To show that $\rho_r(M)<4k+1$, it suffices to show that the second largest zero of $\det(xI_6-M)$ is less then $4k+1$.
By Lemma~\ref{lem2.9}, the second largest zero  of $\det(xI_6-M)$ is $\max\{\alpha_1+\beta_2,\beta_1+\alpha_2\}$, where $\alpha_1\geq \alpha_2$ and $\beta_1\geq \beta_2\geq \beta_3$ are zeros of $\det(xI_2-M_1)$ and $\det(xI_3-M_2)$ respectively.
It is immediate to find $\alpha_2<0<\alpha_1\leq c= 2k+1.$ Since the characteristic polynomial $\det(xI_3-M_2)$ of $M_2$ is   $$h(x)=x^3-(2k-1)x^2+(-5k-1+s+a)x-k(2a+s-3k-1),$$
 $h(0)=-k(2a+s-3k-1)> 0$, $h(2k)=k(-3k-1+s)<0$, and $h(2k+1)=s(k+1)+(k+1)^2+a>0$. Therefore $\beta_1, \beta_2$, $\beta_3$ satisfy   $\beta_3<0<\beta_2<2k<\beta_1<2k+1$.
Hence $\max\{\alpha_1+\beta_2,\beta_1+\alpha_2\}<4k+1,$ and the proof is completed.  \qed

\section*{Acknowledgments}
This research is supported by the National Science and Technology Council of Taiwan R.O.C. under the projects NSTC 113-2115-M-153-001-MY3 and NSTC 113-2115-M-A49-
011-MY2.

\end{document}